\documentclass[12pt,leqno,twoside]{amsart}
\usepackage[utf8]{inputenc}
\usepackage[T1]{fontenc}
\usepackage[colorlinks=true, pdfstartview=FitV, linkcolor=blue, citecolor=blue, urlcolor=blue]{hyperref}
\usepackage{amstext,amsmath,amscd, bezier,indentfirst,amsthm,amsgen,enumerate, geometry,mathrsfs}
\usepackage[all,knot,arc,import,poly]{xy}
\usepackage{amsfonts,color, soul}  
\usepackage{amssymb}
\usepackage{latexsym}
\usepackage{epsfig}
\usepackage{graphicx}
\usepackage{srcltx}
\usepackage{enumitem,tikz-cd}

\newtheorem{theorem}{Theorem}[section]
\newtheorem{corollary}[theorem]{Corollary}

\theoremstyle{definition}
\newtheorem{definition}[theorem]{Definition}
\theoremstyle{remark}
\newtheorem{remark}[theorem]{\sc Remark}
\newtheorem{example}[theorem]{\sc Example}

\renewcommand{\Box}{\square}    

\newcommand{\Sing}{{\mathrm{Sing\hspace{2pt}}}}

\newcommand{\rank}{{\mathrm{rank\hspace{1pt}}}}
\newcommand{\Disc}{{\mathrm{Disc\hspace{2pt}}}}

\newcommand{\id}{{\mathrm{id}}}
\newcommand{\im}{{\mathrm{Im}}}

\newcommand{\Fib}{{\mathrm{Fib}}}

\newcommand{\e}{\varepsilon}
\newcommand{\m}{\setminus}

\newcommand{\fin}{\hspace*{\fill}$\Box$\vspace*{2mm}}

\newcommand{\cF}{{\mathcal F}}

\newcommand{\cS}{{\mathcal S}}
\newcommand{\cQ}{{\mathcal Q}}

\newcommand{\cW}{{\mathcal W}}

\newcommand{\bR}{{\mathbb R}}
\newcommand{\bC}{{\mathbb C}}
\newcommand{\bK}{{\mathbb K}}

\begin{document}

\title[Fibrations of tamely composable maps]{Fibrations of tamely composable maps}
\author{\sc Ying Chen}
\address{School of Mathematics and Statistics, HuaZhong University of Science and Technology WuHan 430074, P. R. China}
\email{ychenmaths@hust.edu.cn}

\author{\sc Cezar Joi\c{t}a}
\address{Institute of Mathematics of the Romanian Academy, P.O. Box 1-764,
 014700 Bucharest, Romania.} 
\email{Cezar.Joita@imar.ro}

\author{Mihai Tib\u{a}r}
\address{Univ. Lille, CNRS, UMR 8524 -- Laboratoire Paul Painlev\'e, F-59000 Lille,
France}
\email{mihai-marius.tibar@univ-lille.fr}

\subjclass[2010]{14D06,  58K05, 57R45, 14P10, 32S20, 32S60, 58K15, 32C18}

\keywords{real map germs,  composed maps, fibrations}

\thanks{Ying Chen acknowledges the support from the National Natural Science Foundation of China (NSFC) (Grant no. 11601168). Cezar Joi\c{t}a acknowledges support from GDRI ECO Math. Mihai Tib\u ar acknowledges  support from the Labex CEMPI (ANR-11-LABX-0007-01). }

\begin{abstract}
 We study composed map germs with respect to their local fibrations. Under most general conditions, inspired by the tameness condition that was introduced recently, we prove the existence of singular tube fibrations, and we determine the topology of the fibres.
\end{abstract}

\maketitle

\section{Introduction}

The existence of a fibered structure on some region of the space has deep impact in mathematics and physics.
John Milnor \cite{Mi} proved in 1968 that a holomorphic function induces a locally trivial fibration in the neighbourhood of a singular point. He 
 launched the challenge of finding the conditions under which complex or real maps may induce such fibrations.

Fibrations of map germs have been considered by many authors ever since, and a lot of valuable knowledge has been amassed
in more and more cases, and in increasing generality.

In this framework, the problem under what conditions the composition of two map germs  
 may be endowed with a local  fibration has been a long term project. 
One first considered the significant setting of a function of type $f\oplus g$,  
where $f$ and $g$ are holomorphic function germs in separate variables and with isolated singularities,  which can be viewed as the composition $G\circ (f,g)$, where $G(u,v) := u+v$ is the simplest linear function.
Let us recall that the original result by Sebastiani and Thom \cite{ST}  determines in this case the topology of the Milnor fibre, and  shows that the monodromy is the tensor product of the monodromies of $f$ and $g$. This was the spring of a stream of studies and far-reaching generalisations e.g. \cite{Sa}, \cite{Ga}, \cite{N1, N2},  \cite{Ba}, \cite{Il}, \cite{HM}, etc,  becoming a \emph{principle} in higher categories,  e.g.  \cite{Ma2}, \cite{DL}, \cite{Le}.

A new remarkable outcome occurred as  N\' emethi  considered the so-called ``composed functions''  in \cite{Ne}, i.e. functions of the form $H :=G\circ F$, where $G$ is a polynomial of 2 complex variables and $F := (f,g)$ is an ICIS, with $f,g$  holomorphic function germs in mixed variables.  Recalling Sakamoto's join result for non-isolated singularities \cite{Sa2}, this was the first time when in the composition $G\circ F$ the map $F$ had a possibly positive dimensional discriminant (in this case of dimension $\le 1$).  N\' emethi  studies in \cite{Ne}  the homotopy type of the Milnor fibre, generalising the join construction, expressing the zeta-function of the monodromy in terms of the zeta-functions of  $f$ and $g$,  and the multivariable Alexander polynomial of $G$.

 More recently, one has studied  in  \cite{PT} the existence of tube fibrations for the real composed maps of type $G(f,g)$ where $f, g$ are holomorphic function germs, and $G=u\bar v$,  by focussing on the Thom condition at the zero set $G(f,g)=0$, in case of non-isolated singularities.
Still in the real setting,  Inaba \cite{In} establishes a general join theorem 
for $f\oplus g$, where $f$ and $g$ are real map germs with zero-dimensional discriminant which are assumed to have local fibrations in the sense of \cite{ACT}.


We have addressed in \cite{CT} the problem of finding conditions under which the composition of two real map germs with positive dimensional discriminant  induces a local singular tube fibration.  The generality consists in including the discriminant in the fibrations, and constructing in this way  a bunch of singular fibrations, and was started recently by the papers  \cite{ACT}  and  \cite{JT}.

Let us point out first that the existence of fibrations sets new challenges in the real setting, some of which have been treated in \cite{ACT}, \cite{JT}. In particular we have found in \cite{JT} a general condition,  called ``tameness'', which solves the local image problem for maps, and in the same time provides local singular tube fibrations.  By \emph{singular tube fibration} we mean that we find stratifications of the source and the target of a map such that over each stratum in the target we have a locally trivial stratified fibration. In particular we include the discriminant of the map in this multi-fibration structure.

\vspace{3pt}

Our paper builds on this general idea of multi-fibration. 
We introduce here a new condition called ``tamely composable'', inspired by the tameness and by the proof of \cite[Theorem 3.2]{CT}, in order to insure the existence of the singular tube fibration for the composed map $H = G\circ F$ in the most general stratified setup.
After proving our existence result Theorem  \ref{t:composed}, we  study the topology of the fibre of  the composed map $H = G\circ F$.
Upon  convenient choices of Milnor data for each of the tube fibrations of $F$, of $G$ and of $H$ (cf \S \ref{ss:Milnordata}), 
two more problems persist: 

(1). The image by $F$ of the fibre $H^{-1}(a)$ might not contain the fibre of $G$.

(2).  The fibre $H^{-1}(a) = F^{-1}(G^{-1}(a))$  contains the pull-back by $F$ of the fibre of $G$ but it is not equal to it.

Problem (1) may happen when $\im F$ is not open as a set-germ at $0$, like in case of the very simple map germ $(x,y) \mapsto (x, xy)$. This problem would be therefore solved if $F$ is a locally open map. The class of locally open maps has been characterised recently in  \cite{JT2} by an algebraic-analytic condition which had been conjectured by Huckleberry in 1971, cf \cite{Hu}. 

Problem (2) is considered here and needs several steps. First of all we need a very careful construction of the stratifications of the maps; this is carried out in Section \ref{s:tamely}.  In Section \ref{topology} we give the topological structure of the stratified fibre of $H$. Finally we describe how our results apply in Nemethi's setup \cite{Ne}.

\section{Preliminaries on  tame map germs}\label{s:mapgerms}

We recall here a few definitions and properties that we need for building our main results.
 \subsection{$\rho$-regularity and tame map germs}\label{s:nmg}

Let  $G:(\bR^{m},0) \rightarrow (\bR^{p}, 0)$ be a non-constant  analytic map germ, $m \ge p \ge1$.
Let $U \subset \bR^m$ be a manifold,  and let	
 $$M(G_{|U}):=\left\lbrace x \in U \mid \rho_{|U} \not\pitchfork_x G_{|U} \right\rbrace $$
	be the set of \textit{$\rho$-nonregular points} of $G_{|U}$, or \emph{the Milnor set of $G_{|U}$}, where
	  $\rho := \| \cdot \|$ denotes here the Euclidean distance function, and $\rho_{|U}$ is its restriction to $U$.
	
	    It turns out from the definition that $M(G_{|U})$ is real analytic.
	     In the following we will actually consider the germ at 0 of $M(G_{|U})$.
	By definition $M(G_{|U})$ coincides with the singular set  $\Sing (\rho, G)_{|U}$ defined in its turn as
the set of points $x\in U$ such that  either  $x\in \Sing(G_{|U})$,   or $x\not\in \Sing(G_{|U})$ and $\rank_{x}( \rho_{|U}, G_{|U}) = \rank_{x}(G_{|U})$.

\begin{definition}[\emph{The Milnor set in the stratified setting}]\label{d:Mstr} \
Let  $G:(\bR^{m},0) \rightarrow (\bR^{p}, 0)$, with $m\geq p >1$, be a non-constant  analytic map germ. We say that  a finite semi-analytic Whitney (a)-regular stratification $\cW$ of $\bR^{m}$ is a stratification of $G$ if $\Sing G$ is a union of strata, and  such that the restriction $G_{|W}$ has constant rank  for any $W\in \cW$.

Let $W \in \cW$,   tacitly understood as the germ at 0 of the stratum $W$, and let
$M(G_{|W})$  
be the Milnor set of $G_{|W}$, as defined above.
One calls 
$$M_{\cW}(G):=\bigsqcup_{W\in \cW} M(G_{|W_\alpha})$$
 the set of \textit{stratwise $\rho$-nonregular points} of $G$ with respect to the stratification $\cW$.
\end{definition}

 By definition, if $\rank G_{|W}= \dim W$, then $W\subset M_{\cW}(G)$. Notice that the Milnor set  $M_{\cW}(G)$  is  closed, due to the Whitney (a)-regularity of the stratification.

By Milnor's classical result on the local conical structure of semi-analytic sets \cite{Mi}, there exists $\e_{0}>0$ such that the manifold $G^{-1}(0) \m \Sing G$ is transversal to the  sphere $S^{m-1}_{\e}$ centred at 0,  for any $0<\e <\e_{0}$. For any fixed point $a\in G^{-1}(0) \m \Sing G$,  a whole open ball $B$ centred at $a$ does not intersect  $\Sing G$, and it then follows that the nearby fibres of $G$ inside $B$
  are also transversal to the levels of the distance function $\rho$, provided that $B$ is small enough.
This implies that $M_{\cW}(G)\cap (G^{-1}(0) \m \Sing G) =\emptyset$,
which proves the following inclusion (see also \cite{JT, CT}):
\begin{equation}\label{r:milnorset}
 M_{\cW}(G)\cap G^{-1}(0) \subset \Sing G \cap G^{-1}(0).
\end{equation}

\begin{definition}[\emph{Tame map germs, \cite{JT}}]\label{d:tame} \ \\
Let $G:(\bR^{m},0) \rightarrow (\bR^{p}, 0)$, with $m > p \ge 2$, be a non-constant  analytic map germ. We say that $G$ is \emph{tame} with respect to the stratification $ \cW$ if the following  inclusion of set germs holds:
\begin{equation}\label{eq:main2}
\overline{M_{\cW}(G) \m  G^{-1}(0)} \cap G^{-1}(0) \subset     \{ 0\}.
\end{equation}
 \end{definition}

It follows from the definition that if $G$ is tame then the closure of the strata $W$ of $\cW$ 
such that $\rank G_{|W}= \dim W$ intersect  $G^{-1}(0)$ at $\{ 0\}$ only. 
Notice also that the maps $G$ with fibre $G^{-1}(0) = \{0\}$ are tame.

The existence of the images as set germs is insured by the following result:

\begin{theorem}\label{main-new}\cite{JT}
Let  $G:(\bR^{m},0) \rightarrow (\bR^{p}, 0)$, with $m \geq p >1$, be a non-constant  analytic map germ.
If $G$ is tame  with respect to the stratification $\cW$ then:
\begin{enumerate}
 \rm \item \it   $\im G$ and  $\Disc(G) := G(\Sing G)$ are well-defined as set germs.
 \rm \item \it For any stratum $V \in  \cW$, the image  $G(V)$ is a well-defined set germ at the origin.
\end{enumerate}
\fin
\end{theorem}

\begin{remark}\label{r:loja}
 Let us point out here that the images of strata $G(V)$ are well-defined set-germs by Theorem \ref{main-new}, thus they are subanalytic sets, and in particular they are triangulable, by the classical result of \L ojasiewicz \cite{Lo}.  This fact will be used in the next sections.
\end{remark}

\subsection{Singular stratified fibration theorem}\label{s:sing}


We recall here the \emph{tame} condition. It turns out that this is the most handy and general condition  under which  one can prove the existence of a local singular fibration.
We consider the general case $\dim \Disc G >0$, and we refer to \cite{JT} for details.

\begin{definition}[\emph{Regular stratification}]\label{d:regularstr}\ 
Let  $G:(\bR^{m},0) \rightarrow (\bR^{p}, 0)$ be a non-constant  analytic map germ, $m> p >1$, and let
$ \cW$ be a Whitney (b)-regular stratification  of $G$ at 0, as defined above.
We assume that $G$ is  tame with respect to $\cW$.
Then  Theorem \ref{main-new}   tells that the
images of all strata of $ \cW$ are well-defined as set germs at 0.
By using the classical stratification theory,
there exists a germ of a finite subanalytic stratification $\cS$   of  the target such that  $\Disc G$ is a union of strata,  and that $G$ is  a stratified submersion relative to the couple of stratifications $( \cW,  \cS)$,
 meaning that the image by $G$ of a stratum  $W_\alpha \in  \cW$ is  a single stratum $S_{\beta} \in \cS$,
	and that the restriction $G_{|}:W_\alpha \to S_{\beta}$ is a submersion.
	
One then calls  $( \cW,  \cS)$ a \emph{regular stratification of the map germ $G$.}
\end{definition}		

\begin{definition}[\emph{Singular Milnor tube fibration}]\label{d:fib}
Let $G:(\bR^{m},0) \rightarrow (\bR^{p}, 0)$, $m\ge p>1$,  be a non-constant analytic map germ. Assume that there exists some regular stratification $( \cW,  \cS)$ of $G$.

We say that $G$ has  a \emph{singular Milnor tube fibration} relative to $( \cW,  \cS)$  if for any small enough $\e > 0$ there exists  $0<\eta \ll \e$ such that the restriction:
\begin{equation}\label{eq:tube1}
G_| :  B^{m}_{\e} \cap G^{-1}( B^{p}_\eta \m \{ 0\} ) \to  B^{p}_\eta \m \{ 0\}
\end{equation}
is a stratified locally trivial fibration which is independent, up to stratified homeomorphisms, of the choices of $\e$ and $\eta$.
By \emph{stratified locally trivial fibration} we mean that for any stratum $S_{\beta}$ of $ \cS$, the restriction $G_{| G^{-1}(S_{\beta})}$ is a locally trivial \emph{stratwise fibration}.
\end{definition}

By ``independent, up to stratified homeomorphisms, of the choices of $\e$ and $\eta$'' we mean that
when replacing $\e$ by some $\e'<\e$, and $\eta$ by some small enough $\eta'<\eta$, then the  map  \eqref{eq:tube1} and its analogous map for $\e'$ and $\eta'$
have the same stratified image in the smaller ball $B^{p}_{\eta'} \m \{ 0\}$, and the corresponding singular fibrations are stratified homeomorphic.

\begin{theorem} \label{t:tube}\cite{JT}
	Let $G:(\bR^m, 0) \to (\bR^p,0)$, $m > p\ge 2$,  be a non-constant analytic map germ. 	
	If $G$ is tame, then $G$ has a singular Milnor tube fibration \eqref{eq:tube1}.
\end{theorem}

  We refer to \cite{JT} for examples and for the relation between \emph{tame} and the Thom regularity, namely it is shown in  \cite{JT}: \emph{if the Whitney stratification $\cW$  is Thom regular at all the strata  included in $G^{-1}(0)$  then $G$ is tame}.

\section{Tamely composable maps}\label{s:tamely}
Several problems arise if one wants that the composition of map germs $H := G\circ F$  has a tube fibration.
First of all we need to choose stratifications such that we have a convenient junction of $F$ with $G$.

 \subsection{Construction of regular stratifications adapted to the composition of maps} \label{ss:stratif}\ 
 
Let $F:(\bR^m,0)\rightarrow (\bR^{p},0)$ and  $G: (\bR^{p},0)\rightarrow (\bR^k,0)$ be map germs that we want to compose.
Let $(\cW',\cQ')$ be a regular stratification of $F$ (Definition \ref{d:regularstr}), where 
$\cW'$ is a Whitney (b)-regular stratification at $0\in \bR^{m}$.
We assume that $F$ is locally open,  and that $F$ is tame with respect to $\cW'$. 
The construction is done in several steps.

\smallskip
  
\noindent \underline{Step 1.} We refine $\cQ'$ to a Whitney (b)-regular stratification at $0\in \bR^{p}$, denoted by $\cQ$,  such that the restriction of $G$ to each stratum of $\cQ$
has constant rank and that $\Sing(G)$ is a union of strata. This implies that $\cQ$ is a stratification of the source of $G$. 

\smallskip

\noindent \underline{Step 2.} 
We consider the pull-back of the strata of $\cQ$ by $F$, and we obtain  a refinement $\cW$ of $\cW'$.
On one hand this refinement preserves the property of Whitney (b)-regularity, and on the other hand $F$ remains tame 
  with respect to $\cW$ because the new strata are pull-backs by $F$ of submanifolds in the target. Indeed, let $P\subset \bR^p$ and $Q\subset \bR^q$ be submanifolds, let $f:P\to Q$ be a submersion, and $S\subset Q$ a submanifold. 
 Let  $f_|:f^{-1}(S)\to S$ denote the restriction of $f$ to the pull-back of $S$. Then the Milnor set $M(f_|)$ is included in $M(f)\cap f^{-1}(S)$. %

\smallskip

\noindent \underline{Step 3.}
We have already shown in Step 2 that $(\cW, \cQ)$ is a regular stratification of $F$.
Now we may construct a stratification $\cS$ at $0\in \bR^{k}$ by the constant rank criterion for the map $G$ as done in Definition \ref{d:regularstr},  such that  $(\cQ, \cS)$ is a regular stratification of $G$.
It then follows that $(\cW, \cS)$ is a regular stratification of $H$.  

\

We have constructed regular stratifications for $F$, $G$ and $H$ adapted to the composition $H = G\circ F$.
From now and until the end we assume that our maps are endowed with such regular stratification adapted to the composition.

  \

 One of the problems, already observed in \cite{CT},  is that the composition of tame maps is not necessarily a tame map. 
 Here is a simple such example.
 \begin{example}
 Let 
 $F:(\bR^4,0)\to(\bR^3,0)$, 
\[F(x,y,u,v)=\bigl((x^2+y^2)(1+u),(x^2+y^2)v, u^2+v^2\bigr). \]
Then $F^{-1}(0)=\{0\}$ and hence $F$ is tame. Let  $G:(\bR^3,0)\to(\bR^2,0)$ be the projection  $G(r,s,t)=(r,s)$, which is also tame.
We get the composition:
\[(G\circ F)(x,y,u,v)=\bigl((x^2+y^2)(1+u),(x^2+y^2)v\bigr),\] 
which was shown in \cite[Example 4.10]{JT} to be not tame.

Compare to Corollary \ref{r:partic} and check that condition \eqref{eq:inclsing} is not fulfilled.
\end{example}

We introduce a new natural condition in the spirit of \eqref{eq:main2}:

\begin{definition}[\emph{Tamely composable maps}]\label{d:F-rhoreg}
 Let $F:(\bR^m,0)\rightarrow(\bR^p,0)$ and $G:(\bR^p,0)\rightarrow(\bR^k,0)$, $m,p,k >0$, be analytic map germs, and consider the composition $H= G\circ F :(\bR^m,0)\rightarrow(\bR^k,0)$.   We say that \emph{$F$ is tamely composable with $G$}  iff:
\begin{equation}\label{eq:incl3}
 \overline{M_{\cW}(H)\setminus H^{-1}(0)}\cap H^{-1}(0) \subset  F^{-1}(0).
\end{equation}
\end{definition}

\smallskip

\begin{remark}\label{r:tame}
\begin{enumerate}
 \item  One observes that condition  \eqref{eq:incl3} is trivially implied by the tameness of $H$. 
\item Condition \eqref{eq:incl3} is equivalent to the following:
\begin{equation}\label{eq:incl37}
 \overline{M_{\cW}(H)\setminus H^{-1}(0)}\cap H^{-1}(0)\cap \Sing H \subset  F^{-1}(0)
\end{equation}
because the left hand side terms of   \eqref{eq:incl3} and of \eqref{eq:incl37} are equal.
Indeed, this is an immediate consequence of the inclusion \eqref{r:milnorset} applied to the map $H$, namely:
\begin{equation}\label{r:milnorset2}
 M_{\cW}(H)\cap H^{-1}(0) \subset H^{-1}(0)  \cap\Sing H.
\end{equation}
\item Condition \eqref{eq:incl3} is also equivalent to the following:
\begin{equation}\label{eq:incl38}
  \overline{F(M_{\cW}(H)) \m  G^{-1}(0)} \cap G^{-1}(0) \subset \{0\}.
\end{equation}
Indeed, by taking the image by $F$ we obtain the implication \eqref{eq:incl3} $\Rightarrow$ \eqref{eq:incl38}. To show the converse, we  take the inverse image of the inclusion \eqref{eq:incl38}. Then the left hand side of the lifted inclusion contains the left hand side of \eqref{eq:incl3}, and this is enough for show that \eqref{eq:incl3} hold too.

In particular, if we take $F = \id$ in \eqref{eq:incl3} or in \eqref{eq:incl38},  then we get back our condition \eqref{eq:main2}.
\end{enumerate}
\end{remark}

\medskip

While condition  \eqref{eq:incl3} does not imply that $G$ is tame, nor the other way around, we prove that this is what we need to add up such that $H$ behaves well:

\begin{theorem}\label{t:composed}
Let $F:(\bR^m,0)\rightarrow(\bR^p,0)$ and $G:(\bR^p,0)\rightarrow(\bR^k,0)$, $m\geq p\geq k\geq 2$, be analytic map germs such that $F$ is tame, and that
 $F$ is tamely composable with $G$.

Then the map germ $H = G\circ F$ is tame,  and has a singular tube fibration.
\end{theorem}

\begin{proof}
We obviously have  $F^{-1}(0)\subset H^{-1}(0)$. By comparing the corresponding Jacobian matrices we deduce the inclusions $M_{\cW}(H)\subset M_{\cW}(F)$. We point out that this inclusion is also due to the choice of regular stratifications adapted to the composition of maps, as we have assumed above.  See also Remark \ref{r:tame}. 


By using these two inclusions, we obtain:
\begin{equation}\label{eq:incl2}
M_{\cW}(H)\setminus H^{-1}(0) \subset M_{\cW}(F)\setminus H^{-1}(0)\subset M_{\cW}(F)\setminus F^{-1}(0).
\end{equation}

Taking closures in the first inclusion of \eqref{eq:incl2}, and intersecting with $H^{-1}(0)$, we obtain the first inclusion in:
\begin{equation}\label{eq:incl4}  \overline{M_{\cW}(H)\setminus H^{-1}(0)}\cap H^{-1}(0) \subset  \overline{M_{\cW}(F)\setminus H^{-1}(0)}\cap H^{-1}(0)\subset  \overline{M_{\cW}(F)\setminus H^{-1}(0)}\cap F^{-1}(0)
\end{equation}
 whereas the second inclusion is a direct consequence of  \eqref{eq:incl3}.

From the last inclusion in \eqref{eq:incl2}, by taking closures,   we get:
\begin{equation}\label{eq:incl5}
\overline{M_{\cW}(F)\setminus H^{-1}(0)}\cap F^{-1}(0) \subset \overline{M_{\cW}(F)\setminus F^{-1}(0)}\cap F^{-1}(0).
\end{equation}
 Chaining together the above inclusions we obtain:
\begin{equation}\label{eq:incl6} \overline{M_{\cW}(H)\setminus H^{-1}(0)}\cap H^{-1}(0) \subset \overline{M_{\cW}(F)\setminus F^{-1}(0)}\cap F^{-1}(0),
\end{equation}
which shows  that the tameness of $F$ implies the tameness of $H$.

  We may now use Theorem \ref{t:tube} to conclude that the map  $H$ has a tube fibration.
\end{proof}
In case $F$ is tame, its discriminant $\Disc F := F(\Sing F)$ is a well-defined subanalytic set germ.
 The following consequence recovers the setting of \cite[Theorem 3.2]{CT}, where $F$ was tame, with isolated singular value, and $G$ had an isolated singular point at the origin:
\begin{corollary}\label{r:partic}
Let $F:(\bR^m,0)\rightarrow(\bR^p,0)$ and $G:(\bR^p,0)\rightarrow(\bR^k,0)$, $m\geq p\geq k\geq 2$, be analytic map germs such that $F$ is tame and that 
\begin{equation}\label{eq:inclsing}
(\Disc F  \cup \Sing G)  \cap  G^{-1}(0)  \subset  \{ 0\}
\end{equation}

Then the map germ $H = G\circ F$ is tame,  and has a singular tube  fibration.
\end{corollary}
\begin{proof} 
Condition \eqref{eq:inclsing} implies the inclusion $\Sing H \cap H^{-1}(0) \subset  \Sing H \cap F^{-1}(0)$. We then get :
$$\overline{M_{\cW}(H)\setminus H^{-1}(0)}\cap H^{-1}(0) \subset \Sing H \cap H^{-1}(0) \subset  \Sing H \cap F^{-1}(0)  \subset F^{-1}(0),$$
whereas the first inclusion follows from  \eqref{r:milnorset}, and the third is trivial. This shows that condition \eqref{eq:incl3} holds. 
 
We may now apply Theorem \ref{t:composed} and get the desired conclusion.
\end{proof}

\begin{remark}\label{r:tame-examples}
  If $\Sing F\cap F^{-1}(0) = \{ 0\}$ then $F$ is tame. Indeed, we have $\Sing F\cap F^{-1}(0) = \{ 0\}$, which implies that $F^{-1}(0) \m \{ 0\}$ is a Thom $a_{F}$-regular stratum. It is well-known that the Milnor set $M(F)$ does not intersect
the  positive dimensional Thom $a_{F}$-regular strata of $F^{-1}(0)$, see e.g. the proof of \cite[Proposition 4.2]{ART}.  

 If $\Sing F\cap F^{-1}(0) = \{ 0\}$ and $F^{-1}(0) \neq \{ 0\}$, it  follows  from \cite[Proposition 2.4]{ART} that $F$ is an open map germ. See also \cite[Proposition 2.4]{JT}. 
\end{remark}


\begin{corollary}\label{c:function}
For $\bK := \bR$ or $\bC$, let $F:(\bK^m,0)\rightarrow(\bK^p,0)$ be an analytic tame map germ, and let $G:(\bK^p,0)\rightarrow(\bK,0)$ be an analytic function germ. Then $H = G\circ F$ is  tame and $F$ is tamely composable with $G$.
\end{corollary}

\begin{proof}
Our composed map $H = G\circ F: (\bK^m,0)\rightarrow(\bK,0)$ is an analytic function, and any analytic function is tame.
 Indeed, this is due to the existence of a  Thom $a_{H}$-regular stratification of
$H^{-1}(0)\setminus\{0\}$, which holds for any analytic function, as proved by Hironaka \cite{Hi}. In turn, this implies the $\rho$-regularity of $H$, and since $H$ has isolated critical value, it follows that $H$ is tame. 

Moreover, $F$ is  tamely composable with $G$ simply because the tameness of $H$ trivially implies the
 tamely composable condition \eqref{eq:incl3}, cf. Remark \ref{r:tame}(a).
\end{proof}

\begin{example}\label{c:icis}
Among the classes of maps $F$ which are tame, and  hence verify Corollary \ref{c:function}, are: holomorphic maps $F: (\bC^{n},0) \to (\bC^{k},0)$ defining an ICIS; real analytic maps $F$ such that $\Sing F \cap F^{-1}(0) \subset \{0\}$.
\end{example}

\

\section{What is the fibre of a composed map?}\label{topology}

 We  describe the topology of the fibre of $H = G\circ F$ in case the maps $F$ and $G$  are tamely composable.
Let $F:(\bR^m,0)\rightarrow(\bR^p,0)$ and $G:(\bR^p,0)\rightarrow(\bR^k,0)$, $m\geq p\geq k\geq 2$.   In \S \ref{ss:stratif} we have constructed regular stratifications $(\cW,\cQ)$ for $F$, and $(\cQ,\cS)$ for $G$,  adapted to the composition $H= G\circ F$. We continue to assume that $F$ and $G$ are endowed with regular stratifications adapted to the composition.

\begin{remark}\label{r:tame2}
  If $F$ is tame before the construction  \S \ref{ss:stratif}, then $F$ is also tame with respect to the
newly constructed stratification $\cW$. This is due to the fact that the strata are all pull-backs by $F$ of strata of the target of $F$ (intersected with the appropriate strata of the source). This fact has been used in the beginning of the proof of Theorem \ref{t:composed} for proving the inclusion of Milnor sets.
However, the difference is made by $G$, where if one introduces new strata in its source, these might also introduce new branches of the Milnor set (which are not anymore pull-backs). Therefore we have to assume that $G$ is tame with respect to this final stratification of its source.
\end{remark}

\subsection{Choice of Milnor data}\label{ss:Milnordata}
 
By Theorem \ref{t:tube}, both $F$ and $G$  have singular tube fibrations. Let us give the details in the following.

We choose appropriate Milnor data for the singular tube fibrations, as follows. Let $\e_{2}>0$ be the maximum Milnor ball for $F$, and let $\delta_{2}$ be the maximum Milnor ball for $G$.  Let then $0<\varepsilon_{0}<\varepsilon_{1}<\e_{2}$ and $0<\delta_{0}<\delta_{1}<\delta_{2}$ such that, for $i=0,1$:
\begin{equation}\label{eq:fib1}
 F_{\mid}: B^{m}_{\varepsilon_{i}}\cap F^{-1}(B^{p}_{\delta_{i}}\setminus\{0\})\rightarrow B^{p}_{\delta_{i}}\setminus\{0\}
\end{equation}
is a stratified locally trivial fibration, and moreover, by choosing some $0<\eta\ll\delta_{0}$, that the restrictions:
\begin{equation}\label{eq:fib2}
 G_{\mid}: B^{p}_{\delta_{i}}\cap G^{-1}(B^{k}_{\eta}\setminus\{0\})\rightarrow B^{k}_{\eta}\setminus\{0\},
\end{equation}
are  singular tube fibrations of $G$.

Assuming in addition that the maps $F$ and $G$  are tamely composable, Theorem \ref{t:composed} tells that the composition $H = G\circ F$ has a singular tube fibration with respect to the regular stratification $(\cW,\cS)$, more precisely, that the restrictions:
\begin{equation}\label{eq:fib31}
H_{\mid}:B^{m}_{\varepsilon_{i}}\cap H^{-1}(B^{k}_{\eta}\setminus\{0\})\rightarrow B^{k}_{\eta}\setminus\{0\}
\end{equation}
are locally trivial stratified fibration, for $i=0,1$, and that these two fibrations are stratified isotopic. In particular their fibres are stratified homeomorphic. We may thus use a single notation $\Fib(H_{|V;S})$ for 
 a fibre of $H^{-1}(a)\cap V$ on a stratum $V\in \cW$ over some point $a\in S$ of a stratum $S\in\cS$.  
 
  Let us remark that the fibre $H^{-1} (a)$ of the fibration of $H$ (as provided by Theorem \ref{t:composed}) over some stratum $S\in \cS$ is a singular stratified set. More precisely,  we have the following decomposition of the fibre of $H$ over  some point $a\in S\cap B^{k}_{\eta_{0}}$: 
\begin{equation}\label{eq:pieces}
 B^{m}_{\varepsilon_{i}}\cap H^{-1} (a) = \bigsqcup_{V\in \cW} \Fib(H_{|V;S}),
\end{equation}
 where  $\Fib(H_{|V;S}) := B^{m}_{\varepsilon_{i}}\cap V\cap H^{-1} (a)$, for $i=0, 1$.

 
 Our following result describes the topology of each piece $\Fib(H_{|V;S})$.

\begin{theorem}\label{t:Fib}
 Let  $(\cW,\cQ)$ and $(\cQ, \cS)$ be regular stratifications of $F$ and $G$, respectively, adapted to the composition $H= F\circ G$. Let $F$ and $G$ be tame, and let $F$ be locally open, and tamely composable with $G$.
 
Then, for any $V\in \cW$, $S\in \cS$,
 $\Fib(H_{|V;S})$ is homotopy equivalent to the total space of a locally trivial fibration of fibre
$\Fib(F_{|V,F(V)})$  over the base space $\Fib(G_{| F(V);S})$.
\end{theorem}

\begin{proof}
By our hypotheses, the map germs $F$, $G$ and $H$ have stratified tube fibrations, and we will  use the notations established above for these fibrations.  We decompose the proof of the theorem in three steps. In order to simplify the notations, we omit to write that all the spaces are intersected with either $V$ or $F(V)$, correspondingly.

 \smallskip

\noindent \textbf{Step 1}. By the fact that $F$ is locally open, and due to the choice of the Milnor data, we have the following inclusions:
$$B^{p}_{\delta_{0}}\hookrightarrow F(B^{m}_{\varepsilon_{0}})\hookrightarrow B^{p}_{\delta_{1}}\hookrightarrow F(B^{m}_{\varepsilon_{1}}).$$
By intersecting with the fibre $G^{-1}(a)$ for $0<\|a\|\ll\eta$, we get:

\begin{equation}\label{eq:Inclusion}
B^{p}_{\delta_{0}}\cap G^{-1}(a)\hookrightarrow F(B^{m}_{\varepsilon_{0}})\cap G^{-1}(a)\hookrightarrow 
B^{p}_{\delta_{1}}\cap G^{-1}(a)\hookrightarrow F(B^{m}_{\varepsilon_{1}})\cap G^{-1}(a).
\end{equation}

 \medskip

\noindent \textbf{Step 2}. We consider the following commutative diagram. In order to simplify the notations, we omit to write that all the spaces on the upper row are intersected with  $V$, and that  all the spaces on the lower row are intersected with $F(V)$.

\begin{displaymath}
    \xymatrix@C=.5em{     
     F^{-1}(B^{p}_{\delta_{0}})\!\cap\!    B^{m}_{\varepsilon_{0}}\!\cap\! H^{-1}(a) \ar@{^{(}->}[r] \ar[d]_F &
        B^{m}_{\varepsilon_{0}}\!\cap\! H^{-1}(a)  \ar@{^{(}->}[r]  & 
        F^{-1}(B^{p}_{\delta_{1}})\!\cap\!    B^{m}_{\varepsilon_{1}}\!\cap\! H^{-1}(a) \ar@{^{(}->}[r] \ar[d]_F & 
        B^{m}_{\varepsilon_{1}}\!\cap\! H^{-1}(a) 
        \\
       B^{p}_{\delta_{0}}\! \cap\! G^{-1}(a) \ar@{^{(}->}[r] & F(B^{m}_{\varepsilon_{0}}) \!\cap\! G^{-1}(a)) \ar@{^{(}->}[r] & B^{p}_{\delta_{1}}\! \cap\! G^{-1}(a) \ar@{^{(}->}[r] & F(B^{m}_{\varepsilon_{1}})\!\cap\! G^{-1}(a) 
    }
\end{displaymath}
where the horizontal arrows are inclusions, and the two vertical arrows are locally trivial fibrations defined by the corresponding restrictions of the map $F$, to be compared with \eqref{eq:fib1}.

We will prove that the inclusion in the middle above:
\begin{equation}\label{eq:homotop0}
B^{m}_{\varepsilon_{0}}\cap H^{-1}(a)  \hookrightarrow 
        F^{-1}(B^{p}_{\delta_{1}})\cap    B^{m}_{\varepsilon_{1}}\cap H^{-1}(a)
\end{equation}
 is a homotopy equivalence. 
 
Let us show that the inclusions:
\begin{equation}\label{eq:homotop1}
 \alpha : F^{-1}(B^{p}_{\delta_{0}})\cap    B^{m}_{\varepsilon_{0}}\cap H^{-1}(a)  \hookrightarrow 
        F^{-1}(B^{p}_{\delta_{1}})\cap    B^{m}_{\varepsilon_{1}}\cap H^{-1}(a)
\end{equation}
and
\begin{equation}\label{eq:homotop2}
\beta : B^{m}_{\varepsilon_{0}}\cap H^{-1}(a)  \hookrightarrow  B^{m}_{\varepsilon_{1}}\cap H^{-1}(a) 
\end{equation}
are homotopy equivalences. 

The inclusion \eqref{eq:homotop2} is a stratified homeomorphism since both sides are fibres of $H$ in the fibrations \eqref{eq:fib31}. It is the stratified homeomorphism well-defined by the flow produced by the distance function  when rescaling the balls in the framework of our chosen Milnor data.

In case of  the inclusion \eqref{eq:homotop1}, let us consider the square containing the two vertical arrows denoted by $F$  in the above commutative diagram. These are restrictions of the tube fibrations of $F$ over the fibre $G^{-1}(a)$, and have the same fibre, which is the fibre of $F$ in a tube fibration \eqref{eq:fib1}, respectively.   By using the long exact sequence of homotopy groups of the fibrations, and the morphism between them, we obtain that the inclusion \eqref{eq:homotop1} induces a weak homotopy equivalence. Since both spaces are
triangulable subanalytic sets, cf \L ojasiewicz \cite{Lo}, see also our Remark \ref{r:loja}, they are CW-complexes and therefore 
 it follows from Whithead's Theorem that the weak homotopy \eqref{eq:homotop1} is a homotopy equivalence.
 
 \medskip

\noindent \textbf{Step 3}.  By \cite[Lemma 3.2]{Ti}, if $\alpha$ and $\beta$ are homotopy equivalences in the 4-terms sequence of inclusions, then it follows that  the inclusion \eqref{eq:homotop0} in the middle is a homotopy equivalence too.

We conclude that the fibre $\Fib(H) =    B^{m}_{\varepsilon_{0}}\cap H^{-1}(a) $ is homotopy equivalent to the 
space $F^{-1}(B^{p}_{\delta_{1}})\cap    B^{m}_{\varepsilon_{1}}\cap H^{-1}(a)$ which is a locally trivial fibration with fibre $\Fib(F) = B^{m}_{\varepsilon_{1}}\cap F^{-1}(a)$ over a base space which is $\Fib(G) = B^{p}_{\delta_{1}}\cap G^{-1}(a)$. This ends our proof.
\end{proof}

\subsection{Application to Nemethi's setup \cite{Ne}}\label{ss:nemethi-case}
N\' emethi considers in \cite{Ne} the composition of a holomorphic function germ $G: (\bC^{2},0) \to (\bC, 0)$ with a holomorphic map germ $F =(f,g): (\bC^{n+1},0)\to (\bC^{2},0)$ which defines an ICIS.  Then its singular locus $\Sing F$ is 1-dimensional, and its discriminant  $\Delta = F(\Sing F)$ is a plane curve germ.

The composition $H = G\circ F$ is a holomorphic function. As remarked in Corollary \ref{c:icis}, it follows that $F$ is locally open, that $F$ is tamely composable with $G$, and so, by Theorem \ref{t:composed}, that $H$ is tame and has a singular tube fibration.

Let us show how Theorem \ref{t:Fib} recovers  Nemethi's \cite[Theorem A(a)]{Ne} on the topology type of the fibre of $H = G\circ F$ in this special setup.

The stratification of the target $\bC$ is trivial, with the origin $0$ and its complement as only strata.
%
In $\bC^{2}$ the following strata are defined: the origin is the stratum of dimension 0, the branches of  $\Delta\m \{0\}$ and the branches of the singular set $\Sing G \subset G^{-1}(0)$ without the origin are the strata of dimension 1; the complement of all these is the stratum of dimension 2.  This stratification is Whitney (b)-regular, denoted by  $\cQ$ in the setup of \S\ref{ss:stratif}.
  At this point we may remark that $G$ has a singular tube fibration with respect to the stratification 
  $\cQ$. Indeed, since the Milnor set $M_{\cQ}(G)$ is a plane curve, the tameness conditions is trivially verified.  
  
 In $\bC^{n+1}$ we have $F^{-1}(\Sing G)$ as union of strata, and the  set $F^{-1}(\Delta)$ as a union of strata. The set $\Sing H \subset H^{-1}(0)$ is of dimension $n-1$ if $\Delta$ intersects $\Sing G$  at the origin only, and of dimension $n$ if this is not the case.   Let us remark that this stratification $\cW$ is Whitney (b)-regular, that $F$ is tame with respect to $\cW$ (see also Remark \ref{r:tame}),  and that $F$ and $G$ are tamely composable as one can easily verify. 
 
For describing the Milnor fibre $H^{-1}(a)$ we only need the strata of the  set $F^{-1}(\Delta)$ which are outside $H^{-1}(0)$.
  The fibre $\Fib(G)$ is a plane curve which intersects the 1-dimensional strata of $\cQ$ at a set of points, call it $A_{G}$, and let $B_{G}:=\Sing F \cap F^{-1}(A_{G})$. 
The fibre  $H^{-1}(a)$ intersects the open stratum $V := \bC^{n+1} \m F^{-1}(\Delta)$ and this is homotopy equivalent to the total space of a fibration with fibre $\Fib (F)$ and base space $\Fib(G) \m \Delta$. Over each point $a_{i}\in A_{G}= \Fib(G) \cap \Delta$ we have a  the fibre $B_{G}\cap F^{-1}(a_{i})$ of $F$ which has isolated singularities.  At each singular point $b_{ij}\in B_{G}\cap F^{-1}(a_{i})$, the fibre $F^{-1}(a_{i})$  is an ICIS of Milnor number denoted by $\mu_{ij}$. 

The fibre $B^{2n+2}_{\varepsilon_{0}}\cap H^{-1}(a) $ decomposes along strata of $\cW$ as in \eqref{eq:pieces}.
We will give now a description of the homotopy type of the entire fibre $B^{2n+2}_{\varepsilon_{0}}\cap H^{-1}(a)$, as follows.
We denote by $\cF$ the general fibre of $F$,  over some point exterior to the 1-dimensional strata of $\cQ$.
The total space of the $F$-fibration over $\Fib(G)$ has this $\cF$ as the generic fibre, and  has singular fibres over each point $a_{i}\in A_{G}$. In the classical theory of complex fibrations with isolated singularities,  the replacement of a generic fibre $\cF_{i} \simeq \cF$ by a singular fibre corresponds to the attaching over $\cF_{i}$ of  $\mu_{ij}$ cells of dimension $n$ for each ICIS singular point $b_{ij}\in B_{G}\cap F^{-1}(a_{i})$, in order to ``kill'' those $(n-1)$-cycles of $\cF_{i}$ which vanish at $b_{ij}$.

Applying our Theorem \ref{t:Fib} to the above setting of complex map germs, we get:

\begin{corollary}\cite[Theorem A(a)]{Ne}
 The Milnor fibre $\Fib(H)$ has the homotopy type of a space obtained from the total
space of a fibre bundle with base space $\Fib(G)$ and with fibre $\cF$
by attaching to it the total number of $N :=\sum_{b_{ij}\in B_{G}}   \mu_{ij}$ cells of dimension $n$.
\fin
\end{corollary}
Let us point out that in practice one needs more data in order to find the homotopy type of the fibre of $H$. We refer to  \cite{Ne} for several examples of (classes of)  hypersurface singularities treated there as composed singularities,  and which had occurred before in the literature.  In particular, Nemethi finds the precise homotopy type of the fibre of $H$ in the case $\Delta \cap G^{-1}(0) = \{0\}$, cf  \cite[Theorem A(b)]{Ne}.


\vspace{\fill}

\end{document}